\documentclass[reqno,12pt]{amsart}

\usepackage[margin=3cm]{geometry}
\usepackage[utf8]{inputenc}
\usepackage{amsmath,amssymb,amsfonts}
\usepackage{faktor} 
\usepackage{xfrac } 
\usepackage{tikz-cd}
\usepackage{enumerate}  \usepackage{stackrel}
\usepackage{tikz-qtree}
\usetikzlibrary{positioning}
\usepackage[colorlinks=true, linkcolor=blue]{hyperref}
\usepackage [english]{babel}
\usepackage [autostyle, english = american]{csquotes}
\MakeOuterQuote{"}

\newtheorem{thm}{Theorem}[section]
\newtheorem{lem}[thm]{Lemma}
\newtheorem{prp}[thm]{Proposition}

\newtheorem*{prob}{Problem}
\theoremstyle{remark}
\newtheorem{rmk}[thm]{Remark}
\newtheorem{ex}[thm]{Example}

\theoremstyle{definition}\newtheorem{dfn}[thm]{Definition}

\newcommand{\N}{\mathbb{N}}

\newcommand{\C}{\mathbb{C}}

\newcommand{\Ka}{K_G}
\newcommand{\Kw}{K_{G,p}}
\newcommand{\Qa}{Q_{E,G}}

\newcommand{\RH}{\mathcal{R}}

\DeclareMathOperator{\Aut}{Aut}

\DeclareMathOperator{\id}{id}
\def\CB{\color{black} }

\DeclareMathOperator{\re}{Re}

   \address{Istituto Nazionale di Alta Matematica ``Francesco Severi'', Research Unit Scuola Normale Superiore, Piazza dei Cavalieri 7, 56126 Pisa, Italy}
\email{{\tt dallara@altamatematica.it}}
     \address{ Dipartimento di Ingegneria Gestionale, dell'Informazione e della Produzione, Universit\`a degli Studi di Bergamo, Viale G. Marconi 5, 24044 Dalmine, Italy.}
\email{{\tt alessandro.monguzzi@unibg.it}}
\medskip

\thanks{{\em Math Subject Classification 2020:} 32A36 .}

\title[Nonabelian ramified coverings and Bergman projection]{Nonabelian ramified coverings and $L^p$-boundedness of Bergman projections in $\mathbb C^2$}

\author[G. Dall'Ara, A. Monguzzi]{G. Dall'Ara and A. Monguzzi}
\date{}

\begin{document}

\begin{abstract}
    In this work we explore the theme of $L^p$-boundedness of Bergman projections of domains that can be covered, in the sense of ramified coverings, by ``nice'' domains (e.g. strictly pseudoconvex domains with real analytic boundary). In particular, we focus on two-dimensional normal ramified coverings whose covering group is a finite unitary reflection group. In an infinite family of examples we are able to prove $L^p$-boundedness of the Bergman projection for every $p\in(1,\infty)$.   
\end{abstract}
\maketitle


\section{Introduction and main results}

\subsection{The \texorpdfstring{$L^p$-boundedness}{\space} problem for Bergman projections} 

Given a domain $D\subseteq \mathbb C^n$, which we always assume to be nonempty and bounded, let $L^2(D)$ be the space of square-integrable functions on $D$ with respect to the Lebesgue measure. We define the Bergman space as \[A^2(D):=L^2(D)\cap \mathcal O(D),\] where $\mathcal O(D)$ is the space of holomorphic functions on $D$.  Then $A^2(D)$ is a closed subspace of $L^2(D)$, a fact that allows to define the Bergman projection $P_D$ as the Hilbert space orthogonal projector
\[
P_D: L^2(D)\to A^2(D).
\]
This is an integral operator, whose integral kernel $K_D(z,w)$ is called the Bergman kernel of the domain $D$. See, e.g., Chapter 1 of \cite{krantz_book} for these classical facts. 

By definition, the projection $P_D$ is bounded with respect to the $L^2$ norm for any domain $D$, whereas it is a hard and challenging problem to understand its behavior with respect to other $L^p$ norms (a problem sometimes called the "$L^p$-regularity problem" in the literature). More precisely, one asks for which $p\in [1,\infty]$ the a priori estimate
\begin{equation}\label{a-priori}
\|P_D f\|_{L^p}\leq C\|f\|_{L^p}\qquad\forall f\in L^2(D)\cap L^p(D)
\end{equation}
holds for some finite positive constant $C$. In this case $P_D$ extends uniquely to a bounded operator $P_D:L^p(D)\to L^p(D)$ (in virtue of the density of $L^2(D)\cap L^p(D)$ in $L^p(D)$). In other words, one is interested in determining the set
\[
\mathcal I(P_D):=\left\{\frac1p\in [0,1]: \eqref{a-priori} \textrm{ holds}\right\}.
\]
By the self-adjunction of $P_D$ on $L^2(D)$  and standard interpolation results, this set is an interval symmetric with respect to the value $\frac1p=\frac12$. Moreover, it is shown in \cite{dallara-L1} that $1\notin \mathcal I(P_D)$ if the domain is smooth, while it does not seem to be known whether this is the case for every bounded domain $D$. 

The set $\mathcal I(P_D)$ is known to be the whole open interval $(0,1)$ when $D$ lies in various classes of smooth pseudoconvex domains \cite{forelli_rudin, phong_stein, NRSW, mcneal_stein,mcneal_singular, charpentier_dupain, lanzani_stein_minimal}, and the literature is rich with examples where $\mathcal I(P_D)$ is strictly contained in $(0,1)$ \cite{krantz_peloso_houston, edholm_pacific, edholm_mcneal, monguzzi} or even degenerates to the singleton $\{\frac12\}$ \cite{barrett_irregularity, zeytuncu, krantz_peloso_stoppato, EM3, krantz_monguzzi_peloso_stoppato}.
We refer the reader to the recent survey paper \cite{zeytuncu_survey} for a more extensive discussion and bibliography. 

Despite the rich literature, several aspects of this fundamental problem are far from being understood. In particular, it would be desirable to better understand \emph{how the presence, and nature, of singular points in the boundary of the domain $D$ is reflected on $\mathcal I(P_D)$}. The recent papers \cite{chen_krantz_yuan_polydisc, bender_chakrabarti_edholm_mainkar} made progress on this question, investigating the $L^p$-boundedness problem for certain classes of singular domains \emph{covered, in the sense of ramified coverings, by polydiscs}. It is the goal of this paper to further develop this theme. 

\subsection{Ramified coverings and previous works}\label{sec-ramified}

Let us begin by making precise what is meant by "ramified covering" in this paper. For us, a \emph{ramified covering} of a bounded domain $D$ consists of the following data:\begin{itemize}
    \item a bounded domain $E\subseteq \mathbb C^n$;
    \item two proper complex analytic subsets $V\subseteq E$ and $W\subseteq D$;
    \item a finite-sheeted holomorphic covering map $\pi:E\setminus V\rightarrow D\setminus W$, that is, a finite-sheeted covering map in the usual sense (see, e.g., Chapter III of \cite{bredon_book}) that is also holomorphic. 
\end{itemize}

In this case, one says that $D$ is covered by $E$ via the map $\pi$. The \emph{degree of the covering} $\pi$ is the number of sheets, that is, the cardinality of the fiber $\pi^{-1}(z)$ over any point $z\in D\setminus W$ (this cardinality is constant in $z$, because $D$, and hence $D\setminus W$, is connected; see, e.g., Prop. 6 of Appendix A6 of \cite{chirka_book} for this fact). 

We sometimes use the suggestive notation $\pi:E\dashrightarrow D$ to denote a ramified covering, stressing the fact that $\pi$ need not be defined on the whole $E$. In fact, $\pi$ admits a holomorphic extension to $E$ (because zero sets of holomorphic functions are removable for bounded holomorphic functions; see, e.g., Appendix A1.4 of \cite{chirka_book}), but notice that the image of the extension is only a subset of the closure $\overline{D}$ in general. 

Our interest on ramified coverings stems from the fact that there are many examples of singular (that is, nonsmooth) domains that can be covered by nonsingular, or less singular, ones via "nice" maps.  

\begin{ex}[from \cite{chen_krantz_yuan_polydisc,bender_chakrabarti_edholm_mainkar}]\label{ex-hartogs} The generalized (rational) Hartogs triangle is the nonLipschitz domain
\[
\mathbb H^{\frac mn}=\{(z_1,z_1)\in\mathbb C^2: |z_1|^\frac mn<|z_2|<1\}, 
\]
where $m,n\geq 1$ are integers. Then $\mathbb H^{\frac mn}$ is covered by the bidisc, via the map $\pi:\mathbb D\times\mathbb D^\times\to \mathbb H^{\frac mn}$ given by
\[
\pi(w_1,w_2)= (w_1 w_2^n, w_2^m).
\]
Here $\mathbb{D}$ is the unit disc and $\mathbb{D}^\times=\mathbb{D}\setminus\{0\}$.
\end{ex}

\begin{rmk}\label{rmk:proper}
If the holomorphic map $\pi: E\to D$ is \emph{proper}, that is, $\pi^{-1}(K)$ is compact whenever $K\subseteq D$ is compact, then $\pi$ defines a ramified covering. More precisely, the "branching locus" is the image of the zero set of the complex Jacobian determinant of $\pi$, i.e., \[W=\pi(\{J(\pi)=0\}),\] and $V=\pi^{-1}(W)$. Here $J(\pi)$ is the determinant of the complex Jacobian matrix of $\pi$. See, e.g., \cite[Theorem 15.1.9, Remark 15.1.10]{rudin_ftub}.
\end{rmk}

To every ramified covering $\pi:E\dashrightarrow D$, it is associated a \emph{covering group} $G$, namely 
\[
G=\{g\in \Aut(E\setminus V): \pi(g.z)=\pi(z)\quad\forall z \in E\setminus V\}.
\]
A ramified covering $\pi:E\dashrightarrow D$ is said to be \emph{normal} if the associated covering group $G$ acts transitively on the fibers of $\pi$, that is, \[
\pi(z_1)=\pi(z_2) \quad \Longleftrightarrow \quad \exists g\in G\colon\ g.z_1=z_2.
\] 
In this case, the cardinality of $G$ equals the degree of the covering.

\begin{rmk}
In the definition of $G$, $\Aut$ may be interpreted either as the automorphism group in the holomorphic category or in the topological category, since the property \[\pi(g.z)=\pi(z)\quad\forall z \in E\setminus V\] automatically forces a continuous $g$ to be holomorphic. Hence, $G$ is the ordinary deck transformation group of the covering $\pi:E\setminus V\rightarrow D\setminus W$. 
\end{rmk}

Let $\pi:E\dashrightarrow D$ be a ramified covering as above. If $P_E$ and $P_D$ denote the Bergman projections on $E$ and $D$ respectively, a result of S. Bell (Theorem 1 of \cite{bell_proper}) implies the following two propositions, showing that the $L^p$-boundedness problem may be "lifted along a ramified covering". Detailed proofs are given in Section \ref{section-lifting-Lp} below (but see also Theorem 4.15 of \cite{bender_chakrabarti_edholm_mainkar}, where a different terminology is employed and also \cite{ghosh2021weighted, ghosh2022toeplitz} for related work). 

\begin{prp}\label{prp_lifting_0}
Let $p\in [1,+\infty)$ and $C>0$. The Bergman projection $P_D$ satisfies the estimate 
\begin{equation*}
\int_D|P_Du|^p\leq C \int_D|u|^p\qquad \forall u\in L^2(D)\cap L^p(D)
\end{equation*}
if and only if the Bergman projection $P_E$ satisfies the weighted estimate
\begin{equation}\label{Bergman_E_0}
\int_E|P_Ev|^p\sigma\leq C \int_E|v|^p\sigma\qquad \forall v\in \pi^*(L^2(D))\cap L^p(E,\sigma),
\end{equation}
where $\sigma:=|J(\pi)|^{2-p}$ and $\pi^*(L^2(D))=\{J(\pi)u\circ \pi\colon\ u\in L^2(D)\}$. 

Moreover, if $\pi$ is normal with covering group $G$, then \begin{equation}\label{description}
\pi^*(L^2(D))=\{v\in L^2(E)\colon \ J(g)v\circ g=v\quad \forall g\in G\}. 
\end{equation}
\end{prp}

When the ramified covering is normal it is possible and convenient to rephrase the restricted weighted estimate \eqref{Bergman_E_0} as an $L^p$-bound on the full $L^2(E)\cap L^p(E,\sigma)$ for an averaged operator. 

\begin{prp}\label{prp:estimate_average}
Let $\pi:E\dashrightarrow D$ be a normal ramified covering with covering group $G$, let $\Pi_G$ be the orthogonal projection  $\Pi_G: L^2(E)\to \pi^*(L^2(D))$ (see \eqref{G-average} below for a formula) and set \begin{equation*}
Q_{E,G}=P_E\circ \Pi_G.
\end{equation*} Then, the  weighted estimate \eqref{Bergman_E_0} is equivalent to the estimate 
\begin{equation}\label{Q_estimate}
\int_{E}|\Qa f|^p\sigma\leq C \int_E |f|^p\sigma\qquad \forall f\in L^2(E)\cap L^p(E, \sigma),
\end{equation}
where $\sigma=|J(\pi)|^{2-p}$. The operator $\Qa$ is an integral operator with kernel given by
\begin{equation}\label{averaged_kernel}
\Ka(z,w)=\frac{1}{|G|}\sum_{g\in G}K(z,g.w)\overline{J(g)}, 
\end{equation}
where $K(z,w)$ is the Bergman kernel of the domain $E$. 
\end{prp}

Notice that in the statements above, as in the rest of the paper, integrals are always with respect to the Lebesgue measure, which is therefore omitted from the notation. 

A first consequence of Proposition \ref{prp_lifting_0} is that \emph{weighted $L^p$ estimates for the Bergman projection on the covering space $E$ yield $L^p$ estimates for the Bergman projection on $D$}. More precisely, if we denote by $\mathcal{C}_p(E)$ the class of nonnegative weights $\sigma\in L^1_{\textrm{loc}}(E)$ for which the weighted inequality 
\begin{equation*}
\int_E|P_Ev|^p\sigma\leq C_\sigma \int_E|v|^p\sigma\qquad \forall v\in L^2(E)\cap L^p(E,\sigma)
\end{equation*}
holds, then a sufficient condition for the $L^p$-boundedness of $P_D$ is the membership of $|J(\pi)|^{2-p}$ in $\mathcal{C}_p(E)$. Hence, an effective strategy in studying the $L^p$-boundedness problem for the Bergman projection on $D$ consists in combining Proposition \ref{prp_lifting_0} with the full or partial characterizations of $\mathcal{C}_p(E)$ existing in the literature for various kinds of "nice" domains (see, e.g., \cite{bekolle_bonami,bekolle_ball, bekolle_canadian, huo_wagner_wick_certain, huo_wagner_wick}). This is indeed the approach of L. Chen, S. Krantz and Y. Yuan \cite{chen_krantz_yuan_polydisc}, who considered domains covered by polydiscs via rational maps, proving in particular the nontrivial fact that $\mathcal I(P_D)$ is always an interval with nonempty interior. While we believe that this method could be fruitfully pursued even in more general contexts (in view, in particular, of the very recent results \cite{huo_wagner_wick_certain, huo_wagner_wick}, already cited above), its major limitation is that it ignores the algebraic structure of the "restricted" space of functions $\pi^*(L^2(D))\cap L^p(E,\sigma)$ appearing in \eqref{Bergman_E_0}. As a consequence, one cannot expect this approach to provide sharp results, but only bounds on the endpoints of the interval $\mathcal I(P_D)$.

A more refined strategy should then take into account the structure of the function space $\pi^*(L^2(D))$, a task that appears to be easier when the covering is normal, thanks to the neat description given by identity \eqref{description} above. This approach was successfully used by C. Bender, D. Chakrabarti, L. Edholm and M. Mainkar in \cite{bender_chakrabarti_edholm_mainkar}, where they exactly determine $\mathcal I(P_D)$ when $D$ is a monomial polyhedron. These domains (for which see also \cite{nagel_pramanik}) generalize the Hartogs triangles of Example \ref{ex-hartogs}, being covered by polydiscs via general monomial maps. 

The normal ramified coverings of \cite{bender_chakrabarti_edholm_mainkar} are always abelian, in the sense that the covering group $G$ is a (finite) abelian group. Hence, \emph{it is natural to ask whether a similar sharp result, namely a complete description of $\mathcal{I}(P_D)$, may be obtained for interesting classes of domains $D$ that admit nonabelian ramified coverings}. We now turn our attention to a large class of domains of this kind.

\subsection{A class of nonabelian ramified coverings}\label{sec:new}

Let $G$ be a \emph{finite unitary reflection group} (from now on, f.u.r.g.), i.e., a finite subgroup of the unitary group $U(n)$ that is \emph{generated by reflections}. We recall that a (unitary) reflection is an element of $U(n)$ of finite order that fixes pointwise a hyperplane. We refer to \cite{lehrer_taylor} for the theory of f.u.r.g.'s. 

F.u.r.g.'s were completely classified by G.Shephard and J. Todd (\cite{shephard_todd}), who also proved that they are characterized among finite unitary groups by the crucial property that their ring of invariants
\[
\C[z_1,\ldots, z_n]^G=\{P\in \C[z_1,\ldots, z_n]\colon\ P(g.z)=P(z)\quad \forall g\in G\}
\]
is generated by $n$ algebraically independent homogeneous polynomials $P_1,\ldots, P_n$. Thus, for any f.u.r.g.\ $G$ and any choice of generators $P_1,\ldots, P_n$ as above, we can define the so-called $G$-orbit map $\pi:\C^n\rightarrow \C^n$ as follows: \[\pi(z)=(P_1(z),\ldots, P_n(z))\qquad (z\in \C^n).\] 

\begin{ex}[Cf. Chapter 2 of  \cite{lehrer_taylor}] \label{ex:G} Let $m, n\geq1$ and assume that $\ell$ is a positive divisor of $m$.  If $\theta=e^{\frac{2\pi i}{m}}$, the f.u.r.g. $G(m,\ell,n)$ consists of the linear transformations of the form
\[
(z_1,\ldots,z_n)\mapsto (\theta^{\nu_1}z_{\tau(1)},\ldots, \theta^{\nu_n}z_{\tau(n)})
\]
where $\tau$ is a permutation of $\{1,\ldots, n\}$ and the $\nu_j$'s are integers whose sum is divisible by $\ell$. 

The family of groups $G(m, \ell, n)$ includes familiar ones: $G(m,\ell,1)$ is the cyclic group of order $m/\ell$, $G(1,1,n)$ is the symmetric group Sym$(n)$ (acting on $\C^n$ in the standard way), 
$G(m,m,2)$ is the dihedral group of order $2m$.

An orbit map for $G(m,\ell,2)$ is given by
\begin{equation}\label{orbit_map_G}
\pi(z_1,z_2)=\big(z_1^m+z_2^m, (z_1 z_2)^{\frac m\ell}\big).
\end{equation}
\end{ex}

The following elementary proposition provides us with a large class of nonabelian ramified coverings. A proof is given in Section \ref{section-lifting-Lp}.

\begin{prp}\label{rudin_criterion_0}
Let $G$ be a f.u.r.g. and let $\pi$ be a $G$-orbit map as above. Let $E\subseteq \C^n$ be a bounded $G$-invariant domain. Then $D:=\pi(E)$ is a bounded domain and \[\pi:E\rightarrow D\] is a proper holomorphic mapping. Hence, in view of Remark \ref{rmk:proper}, it defines a normal ramified covering, whose covering group is $G$.          \end{prp}

\begin{rmk} Rudin \cite{rudin_refl} actually proved that any proper holomorphic mapping with domain the unit ball $B_n\subseteq \mathbb C^n$ (satisfying a mild boundary regularity assumption) is equivalent, modulo biholomorphisms of the ball and the target domain, to a $G$-orbit map as above. 
\end{rmk}

At this point, it is possible to give a general formulation of the question of main interest in this paper.

\begin{prob} Let $G$ be a f.u.r.g. Let $E$ be a "nice" $G$-invariant domain. Determine $\mathcal{I}(P_D)$, where $D=\pi(E)$ and $\pi$ is a $G$-orbit map.
\end{prob}

By Proposition \ref{prp_lifting_0} it is apparent that the Problem above becomes more difficult as the algebraic structure of the covering group $G$ becomes more complex. In fact, the complexity of $G$ is reflected both in the complexity of the weight $\sigma$ and in that of the space $\pi^*(L^2(D))$. 

Our main result is a "complexity reduction" theorem valid when $E$ and $D$ are two-dimensional. In the forthcoming section, we state this result and show how it yields an optimal solution to the $L^p$-boundedness problem for an infinite family of examples. 

\subsection{The main theorem and an application}\label{section_main_thm}
In order to state our main theorem, we need some more notation and a few more definitions. 

Given a reflection $r\in U(n)$, we let \[
Y_r:=\{v\in \C^n\colon\ r.v=v\}
\]
be its \emph{reflecting hyperplane} (in fact, as recalled above, a reflection is defined by the fact that its fixed point set is a hyperplane). If $G$ is a f.u.r.g., we denote by \[
\mathcal{R}_G=\{Y_r\colon\ r\in G \text{ is a reflection}\}
\] the set of reflecting hyperplanes of $G$. 

If $Y\in \mathcal{R}_G$, then any vector $e$ of norm $1$ orthogonal to $Y$ is called a \emph{root} of the reflection $r$.

The group $G$ acts on $\mathcal{R}_G$ in the natural way: $g.Y:=\{g.v\colon\ v\in Y\}$ when $g\in G$ and $Y\in \mathcal{R}_G$. In fact, $g.Y$ is indeed a reflecting hyperplane of $G$, because if $r$ is a reflection of $G$ and $g\in G$, one has the easily verified identity \begin{equation}\label{G-action} g.Y_r=Y_{grg^{-1}}.\end{equation}

To every $Y\in \mathcal{R}_G$ we associate the subgroup \[
G_Y:=\{g\in G\colon\ g.v=v\quad \forall v\in Y\}, 
\]
which consists of the identity plus all the reflections $r\in G$ such that $Y_r=Y$. The group $G_Y$ is easily seen to be cyclic of some finite order $m_Y$, a quantity that we call the \emph{multiplicity} of $Y$. 

More generally, we may attach to an arbitrary subset $\mathcal{S}\subseteq \mathcal{R}_G$ the reflection subgroup $G_{\mathcal{S}}$ generated by $\bigcup_{Y\in \mathcal{S}}G_Y$, namely the smallest subgroup of $G$ containing all the reflections whose reflecting hyperplane lies in $\mathcal{S}$. It is an immediate consequence of \eqref{G-action} that if $\mathcal{S}$ is $G$-invariant, then $G_{\mathcal{S}}$ is a normal subgroup of $G$. 

We can now describe a key "complexity reduction" procedure, whose importance for our goal will soon be apparent. Given a f.u.r.g. $G$, the set of reflecting hyperplanes $\mathcal{R}_G$ splits as a disjoint union of $G$-orbits \[
\mathcal{R}_G=\mathcal{S}_1\cup \ldots \cup \mathcal{S}_t, 
\]
where possibly $t=1$. Since each $\mathcal{S}_j$ is $G$-invariant, we have a family of normal reflection subgroups $\{G_j:=G_{\mathcal{S}_j}\}_{j=1,\ldots, t}$. The procedure can now be iterated, applying it to each subgroup $G_j$, then to the further subgroups of these that it produces and so on. Of course the iteration stops after a finite number of steps. The result is a tree of subgroups, as the next example shows in the case of the group $G(m,m,2)$ introduced in Example \ref{ex:G} above. 

\begin{ex}\label{ex:G2}
Let $G_m=G(m,m,2)$. We have
\begin{equation}\label{G_m}
G_m=\{r_{\theta_k}, s_{\theta_k}: \theta_k=e^{\frac{2\pi ik}{m}}, k=1,\ldots,m\},\end{equation} where
\begin{align*}
&r_{\theta_k}(z_1,z_2)=(\theta_k z_2, \theta_k^{-1}z_1)\\&s_{\theta_k}(z_1,z_2)=(\theta_k z_1, \theta_k^{-1}z_2).\end{align*}
The group $G_m$ is generated by the reflections $r_{\theta_k}$ and its reflecting hyperplanes are 
\[
Y_{r_{\theta_k}}=\{(z_1,z_2): z_1={\theta_k} z_2\}.
\]
Let now $\theta$ be a primitive $m$th-root of unity and let $j$ be a positive integer. Then
\[
(s_{\theta}) r_{\theta^j}(s_\theta)^{-1}= r_{\theta^{j+2}},
\]
and therefore
\[
s_\theta. Y_{r_{\theta^j}}= Y_{r_{\theta^{j+2}}}.
\]
It is clear that if $m$ is odd, then $\mathcal R_{G_m}$ consists of just one orbit. If $m$ is even, we have two orbits: 
\[ 
\mathcal R_{G_m}=\mathcal R_{G_\frac m2}\cup \mathcal R_{\widetilde G_{\frac m2}}. 
\]
Here $G_{\frac m2}$, defined as in \eqref{G_m}, is the normal subgroup of $G_m$ generated by the reflections \[
\big\{ r_{\theta_k}: \theta_k=e^\frac{2\pi i(2k)}{m}, k=1,\ldots,\textstyle\frac m2\big\},
\]
while $\widetilde G_{\frac m2}$ is the normal subgroup generated by the reflections
\[
\big\{ r_{\theta}: \theta=e^\frac{2\pi i(2k+1)}{m}, k=1,\ldots,\textstyle\frac m2\big\}.
\]
Notice that $\widetilde G_{\frac m2}=h G_{\frac m2}h^{-1}$, where $h(z_1,z_2)=(e^{\frac{2\pi i}{m}}z_2,z_1)$. In particular $\widetilde G_{\frac m2}$ and $G_{\frac m2}$ are conjugated as subgroups of $U(2)$ (of course not as subgroups of $G_m$, since both are normal). 

Thus, if $m=2^kd$ with $d$ odd, we can iterate the procedure $k$ times, obtaining $2^k$ pairwise $U(2)$-conjugate subgroups $G_1,\ldots, G_{2^k}$ isomorphic to $G_d$. See figure \ref{figure_tree}. At this point no further reduction is possible, since each $\mathcal{R}_{G_j}$ consists of a single orbit.

\CB
\begin{center}
\begin{figure}\label{figure_tree}
\begin{tikzpicture}[node distance=2cm]
\node(Gm)                           {$G_m$};
\node(Gm2)      [below left=0.5cm and 1.5cm of Gm]  {$G_{\frac m2}$};
\node(tGm2)       [below right=0.5cm and 1.5cm of Gm] {$\widetilde G_{\frac m2}$};
\node(Gm4)      [below left=0.5cm and 0.8cm of Gm2]{$G_{\frac{m}{4}}$};
\node(tGm4)      [below right=0.5cm and 0.8cm of Gm2]{$\widetilde G_{\frac{m}{4}}$};
\node(P2)      [below left=0.4cm and 0.4cm of tGm4]{\phantom{$H$}};
\node(P3)      [below right=0.4cm and 0.4cm of tGm4]{\phantom{$H$}};
\node(Gm8)[below left=0.4cm and 0.4cm of Gm4]{$G_{\frac{m}{8}}$};
\node(tGm8)[below right=0.4cm and 0.4cm of Gm4]{$\widetilde G_{\frac{m}{8}}$};
\node(P4)      [below left=0.5cm and 0.8cm of tGm2]       {\phantom{$C_2$}};
\node(P5)      [below right=0.5cm and 0.8cm of tGm2]      {\phantom{$C_2$}};
\draw(Gm)       -- (tGm2);
\draw(Gm)       -- (Gm2);
\draw[dashed](tGm2)       -- (P4);
\draw[dashed](tGm2)       -- (P5);
\draw(Gm2) -- (Gm4);
\draw(Gm2) -- (tGm4);
\draw[dashed](tGm4) -- (P2);
\draw[dashed](tGm4) -- (P3);
\draw(Gm4) -- (Gm8);
\draw(Gm4) -- (tGm8);
\end{tikzpicture}
\caption{The tree of subgroups obtained applying the reduction procedure to $G(m,m,2)$. Here $8$ divides $m$.}
\end{figure}
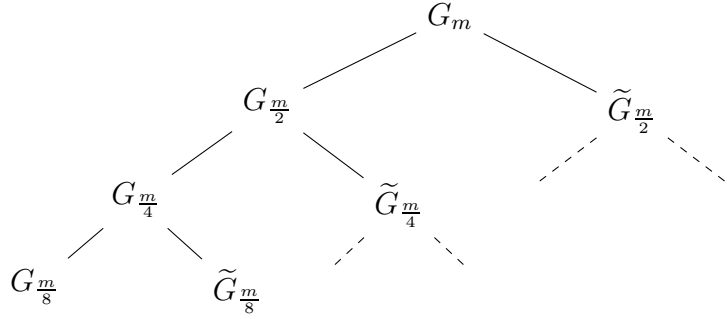
\end{center}\end{ex}
There is one last object we need to associate to a f.u.r.g. $G$. Choosing a root $e_Y$ for each reflecting hyperplane, we define the \emph{Jacobian polynomial} of $G$ as \begin{equation}\label{J_G}
J_G(z)=\prod_{Y\in\mathcal R_G}\left\langle z, e_Y\right\rangle^{m_Y-1}, 
\end{equation}
where $\left\langle\cdot, \cdot\right\rangle$ is the standard Hermitian product of $\C^n$.

The Jacobian polynomial is defined up to a multiplicative constant of unit modulus depending on the choice of roots. Its importance dwells in the fact that $|J_G(z)|=|J(\pi)(z)|$ for any choice of $G$-orbit map $\pi$  (see Proposition \ref{prp_Jacobian} below).

Finally, our main theorem holds under a mild restriction on the covering domain $E$, which we proceed to formulate. 

\begin{dfn}\label{dfn_good_point}
A bounded $G$-invariant domain $E\subseteq\mathbb C^n$ is said to be \emph{good} if for every reflecting hyperplane $Y\in \RH_G$ we have that $Y\cap E$ is connected and $Y\cap bE = b_Y(Y\cap E)$, where $b_Y$ denotes the topological boundary as a subset of $Y$. 
\end{dfn}

\begin{rmk}The unit ball $B_n$ is good with respect to every f.u.r.g. $G$. An example of a domain which is not good with respect to the group of order $2$ generated by $(z_1,z_2)\mapsto (-z_1, z_2)$ is \[
E=\{|z_1-1|^2+|z_2|^2<1\}\cup\{|z_1+1|^2+|z_2|^2<1\}\cup\{|z_1|^2+|z_2-1|^2<0.99\}.
\]
\end{rmk}

With the above notation and definitions, we can finally state our theorem. The main thrust is that it is possible to estimate a variant of the averaged kernel $K_G(z,w)$ in \eqref{averaged_kernel} with a finite sum of analogous kernels where $G$ is replaced by two smaller normal reflection subgroups.  

\begin{thm}\label{Lp_bdd}
Let $G\subseteq U(2)$ be a f.u.r.g. and let $E\subseteq\mathbb C^2$ be a good bounded $G$-invariant domain containing the origin. Assume the following: 
\begin{itemize}
    \item[(i)] the holomorphic Bergman kernel $K(z,\overline w)$ of $E$ can be analytically continued (see Section \ref{section-division-lemma} for the precise definition) past any off-diagonal boundary point of \[
\widetilde{E}=\{(z,\overline{w})\colon z,w\in E\}\subseteq\C^n\times \C^n,\]
where off-diagonal here means $(z,\overline{w})$ with $z\neq w$;
\item[(ii)]  we have two $G$-invariant sets $\mathcal S_1,\mathcal S_2\subseteq\mathcal R_G$ such that
\begin{itemize}
\item[(a)] $\mathcal R_G=\mathcal S_1\cup \mathcal S_2$;
\item[(b)] $\mathcal S_1\cap\mathcal S_2=\emptyset$;
\item[(c)] $\mathcal S_j=\mathcal{R}_{G_j}$ $(j=1,2)$, where $G_j=G_{\mathcal S_j}$ denotes the normal subgroup generated by $\mathcal S_j$.
\end{itemize}
\end{itemize}

If, for $p\in (1,+\infty)$, we set
\begin{equation}\label{kernel_KGp}
K_{G,p}(z,w) = |J_G(z)|^{\frac{2}{p}-1}K_G(z,w)|J_G(w)|^{1-\frac{2}{p}}, 
\end{equation} then there exists a positive constant $C$ such that
\begin{equation}\label{main-estimate}
|K_{G,p}(z,w)|\leq C\frac{1}{|G|}\sum_{g\in G}\left(|K_{G_1,p}(g.z,w)|+|K_{G_2,p}(g.z,w)|+1\right), 
\end{equation}
for every $z,w\in E$. 
\end{thm}

\begin{rmk}\label{rmk:main_thm}
Let $\pi$ be a $G$-orbit map for the group $G$ and set $D:=\pi(E)$. In view of the identity $|J_G(z)|=|J(\pi)(z)|$, the weighted $L^p$ estimate \eqref{Q_estimate} is equivalent to an unweighted $L^p$ estimate for the integral operator with kernel \eqref{kernel_KGp}. This clarifies the relevance of \eqref{main-estimate} for the $L^p$-boundedness problem for $P_D$. 
\end{rmk}

\begin{rmk} 
The hypothesis on the analytic continuation of the Bergman kernel is satisfied whenever the domain $E$ is strictly pseudoconvex and has real analytic boundary. See \cite{franccis_hanges} and the references therein.
\end{rmk}

\begin{rmk}\label{RGS=S}
Unfortunately, we are unaware of any general result allowing to deduce condition (ii-c) above from (ii-a) and (ii-b): the reflection hyperplanes of $G_{\mathcal{S}_j}$ form a $G$-invariant subset of $\mathcal R_G$ containing $\mathcal{S}_j$, but we cannot rule out the possibility that the containment is strict. In principle, a case-by-case check could be carried out using the classification of normal reflection subgroups of f.u.r.g.'s in \cite{arreche_williams}.
\end{rmk}

Let us now see how repeated applications of Theorem \ref{Lp_bdd} provide the promised solution of the $L^p$-boundedness problem for an infinite family of examples.

\begin{thm}\label{thm_application}
Let $B_2$ be the unit ball in $\mathbb C^2$ and let $\pi: B_2\dashrightarrow \pi_k(B_2):=D_k$ be the normal ramified covering defined by
\[
\pi_k(z_1,z_2)=(z_1^{2^k}+z_2^{2^k}, z_1z_2),
\]
for a nonnegative integer $k$. Then $\mathcal I(P_{D_k})=(0,1)$, that is, the Bergman projection of $D_k$ is $L^p$-bounded if and only if $p\in (1,+\infty)$.
\end{thm}

\begin{proof}

As recalled in Example \ref{ex:G}, formula \eqref{orbit_map_G}, the mapping $\pi_k$ is an orbit map for the group $G_{2^k}=G(2^k,2^k,2)$.  

The complexity reduction procedure worked out in Example \ref{ex:G2} for the group $G_{2^k}$ shows that we are in a position to apply Theorem \ref{Lp_bdd} several times and conclude that 
\[|K_{G_{2^k}, p}(z,w)|\leq C\sum_{j=1}^{2^k}\sum_{g\in G} |K_{H_j, p}(g.z,w)| + C, 
\]
where the $H_j$'s are $2^k$ normal subgroups, all conjugate (as subgroups of $U(2)$) to $G_1=G(1,1,2)$. Notice that all the assumptions of Theorem \ref{Lp_bdd} are satisfied in the present case: assumption (i) is clear from the explicit expression for the Bergman kernel of the unit ball, while assumption (ii) is apparent from the discussion in Example \ref{ex:G2}. 

As observed in Remark \ref{rmk:main_thm}, the desired $L^p(D_k)$-boundedness of $P_{D_k}$ is equivalent to the $L^p(B_2)$-boundedness of the operator with integral kernel $K_{G_{2^k}, p}(z,w)$. Thus, our task is reduced to showing the $L^p(B_2)$-boundedness of the operators with positive kernels $|K_{H_j, p}(g.z,w)|$. 

Since the $H_j$'s are pairwise conjugated as subgroups of $U(2)$ and $B_2$ is $U(2)$-invariant, it is easy to see that it is enough to take care of, say, $|K_{G_1, p}(z,w)|$. Notice that \[
G_1=G(1,1,2)=\{s,\text{id}\}
\]
where $s(z_1,z_2)=(z_2,z_1)$. Conjugating this subgroup with $g(z_1,z_2)=(\textstyle\frac{z_1+z_2}{\sqrt{2}},\textstyle\frac{-z_1+z_2}{\sqrt 2})\in U(2)$, we obtain
\[
H=g^{-1}G_1 g=\{t, \text{id}\}
\]
where $t(z_1,z_2)=(-z_1,z_2)$. This group $H$ is the reflection group associated to the normal  ramified covering $\pi_1:B_2\to \pi_1(B_2)$ where $\pi_1(z_1,z_2)=(z_1^2,z_2)$ and, again, the operators with kernel $|K_{H,p}|$ and $|K_{G_1,p}|$ share the same $L^p(B_2)$ mapping properties. The operator with kernel $|K_{H,p}|$ turns out to be $L^p(B_2)$ for every $p\in(1,\infty)$; this is implicitly contained in \cite{bonami_lohoue, bekolle_compte}, but for the reader's convenience we provide a proof of this fact in Appendix \ref{appendix}. Therefore, we conclude that the Bergman projection of $\pi(B_2)$ where $\pi(z_1,z_2)=(z_1^{2k}+z_2^{2k},z_1z_2)$, $k\in\N$, is $L^p$-bounded for all $p\in(1,\infty)$.

In order to complete the proof of Theorem \ref{thm_application}, we still need to show that the Bergman projection is not bounded on $L^1$ or, equivalently, on $L^\infty$. This can be done generalizing a known argument establishing $L^\infty$-unboundedness for the unit ball. It is based on two facts: 

\begin{itemize}
    \item[i)] on every bounded domain $D\subseteq \C^n$ there is an unbounded function $f\in A^2(D)$ such that $\re(f)$ is bounded. 
    \item[ii)] assume further that there is a multi-index $(a_1,\ldots, a_n)\in (\N^{>0})^n$ such that \[
    \zeta\in \overline{\mathbb{D}},\quad  (z_1,\ldots, z_n)\in D\quad \Longrightarrow \quad (\zeta^{a_1}z_1, \ldots,\zeta^{a_n}z_n)\in D,  
    \]
    so that in particular $0\in D$.
 Then, for every $g\in A^2(D)$ vanishing at $0$ we have the identity $P_D(\re(g))=g$. 
\end{itemize}
Applying ii) to $g=f-f(0)$, where $f$ is as in i), we immediately see that any domain satisfying the assumption of ii) has $L^\infty$-unbounded Bergman projection. Since the domains $D_k$ satisfy this assumption with $(a_1,a_2)=(2^k, 2)$, we are reduced to checking i) and ii). 

For i), assume w.l.o.g. that $\inf_{D}\re(z_1)=0$. Then $f(z)=i\log(\re(z_1))$ (principal determination of $\log$) is the desired function, because $|\re(f)|\leq \frac{\pi}{2}$ on $D$. This proves i). 

Let us prove ii). Since $\re(g)=g+\overline{g}$ and $g\in A^2(D)$, the statement follows from the fact that $\overline{g}$ is orthogonal to the Bergman space $A^2(D)$ when $g\in A^2(D)$ vanishes at $0$. Let then $u\in A^2(D)$: we want to show that $\int_D ug=0$. We will check that any $v\in L^1(D)\cap \mathcal{O}(D)$ satisfies the mean value property \begin{equation}\label{mean_value}
\frac{1}{|D|}\int_D v = v(0), 
\end{equation}
where $|D|$ is the Lebesgue measure of $D$. By assumption the map $(z_1,\ldots, z_n)\mapsto (\zeta^{a_1}z_1, \ldots,\zeta^{a_n}z_n)$ is a measure-preserving biholomorphism of $D$ for every $\zeta\in b\mathbb{D}$ in the unit circle. Then \[
\int_D v(z_1,\ldots, z_n) = \int_D v(\zeta^{a_1}z_1, \ldots,\zeta^{a_n}z_n)\qquad \forall \zeta\in b\mathbb{D}.
\] Averaging in $\zeta\in b\mathbb{D}$ and using Fubini, we get \[
\int_D v(z_1,\ldots, z_n) = \int_D \left(\frac{1}{2\pi}\int_{b\mathbb{D}} v(\zeta^{a_1}z_1, \ldots,\zeta^{a_n}z_n)|d\zeta| \right).
\]
Since $\zeta\mapsto v(\zeta^{a_1}z_1, \ldots,\zeta^{a_n}z_n)$ is a holomorphic function on the closed unit disc, the ordinary mean value property gives the desired identity \eqref{mean_value}. 
\end{proof}

Notice that the boundaries of the domains $D_k$ of Theorem \ref{thm_application} are singular at points where the Jacobian determinant $J(\pi_k)$ vanishes. The most remarkable feature of Theorem \ref{thm_application} is probably the fact that, despite the singular nature of the domains $D_k$, their Bergman projection is bounded on the maximal range of $L^p$ spaces. To the authors' knowledge, the only "easy" example of a domain with this property is the polydisc. 

It would be of great interest to know whether the property $\mathcal{I}(P_D)=(0,1)$ is shared by other domains $D$ that are covered by the unit ball $B_2$. Notice that our method is of little or no use when the covering group has few or no normal reflection subgroups. The simplest example is probably $G=G(3,3,2)$ in the Shepard--Todd notation, a $G$-orbit map of which is given by \[
\pi_3(z_1,z_2) = (z_1^3+z_2^3, z_1z_2). 
\]

\subsection{Organization of the paper} In Section \ref{section-lifting-Lp} we prove Propositions \ref{prp_lifting_0} and \ref{prp:estimate_average}. In Section \ref{section-groups} we recall various basic facts about f.u.r.g.'s and their $G$-orbit maps, whereas Sections \ref{section-division-lemma}, \ref{section-normal-lemma} and \ref{section-covering} are devoted to the proofs of three key lemmas from which the main theorem is deduced in Section \ref{section-main}.

\subsection{Acknowledgments}
The authors would like to thank C. E. Arreche and N. F. Williams for some helpful comments on normal reflection subgroups of f.u.r.g.'s. 

The authors would also like to acknowledge the hospitality of Marco M. Peloso and the Department of Mathematics of Università Statale di Milano, where part of the research work was conducted in November 2021. 

The first author would like to acknowledge the financial support of the Istituto Nazionale di Alta Matematica "F. Severi", and the second author the support of the Hellenic Foundation for Research and Innovation (H.F.R.I.) under the “2nd Call for H.F.R.I. Research Projects to support Faculty Members \& Researchers” (Project Number: 4662).

\section{Lifting \texorpdfstring{$L^p$-boundedness}{\space} along a ramified covering}\label{section-lifting-Lp}

In this section we discuss a slight generalization of Bell's transformation rule for the Bergman projection in the setting of ramified coverings. 

We start with a few preliminary remarks. First of all, if $\psi:E_1\rightarrow E_2$ is a holomorphic mapping between two domains in $\C^n$, and $f:E_2\rightarrow \C$ then we have the pull-back operator \begin{equation}\label{pull-back}
\psi^* f:=J(\psi) f\circ \psi,
\end{equation}
where $J(\psi)$ is the complex Jacobian determinant of $\psi$. This is nothing but the natural pull-back on $(n,0)$-forms, under the (coordinate-dependent) identification of the function $f:E_2\rightarrow \C$ with the form $f(z)dz_1\wedge \cdots \wedge dz_n$. 

Next, consider a genuine covering $\pi:E\rightarrow D$ (i.e., one for which $V$ and $W$ are empty, cf. the definition in the Introduction). Then $\pi^*f$ is a well-defined element of $L^2(E)$ for every $f\in L^2(D)$, because of the change of variable formula \begin{equation}\label{change_of_variable}
\int_D g = \frac{1}{d}\int_E g\circ \pi\, |J(\pi)|^2, 
\end{equation} where $d$ is the degree of the covering. Identity \eqref{change_of_variable} holds for every nonnegative Borel function $g$ defined on $D$, and it follows from the standard change of variable formula and the fact that $\pi:E\rightarrow D$ is a $d$-to-$1$ covering map. Hence, \[||\pi^*f||_{L^2(E)} = \sqrt{d}||f||_{L^2(D)}.\] 

Of course, if $f\in A^2(D)$, then $\pi^*f$ is holomorphic, and hence in $A^2(E)$. Since finite-sheeted coverings are proper maps, by Theorem 1 of \cite{bell_proper} we have the commutative diagram 
\[
\begin{tikzcd}[row sep=huge]
L^2(E)  \arrow{r}{P_{E}}  &A^2(E)   \\
L^2(D) \arrow{u}{\pi^*}  \arrow{r}{P_{D}}  &A^2(D) \arrow{u}{\pi^*} 
\end{tikzcd}.
\]

Let now $\pi:E\dashrightarrow D$ be a ramified covering. Since $\pi:E\setminus V\rightarrow D\setminus W$ is a genuine covering, the discussion above yields the extended diagram: 
\[
\begin{tikzcd}[row sep=huge]
L^2(E)\arrow{r}&L^2(E\setminus V)  \arrow{r}{P_{E\setminus V}}  &A^2(E\setminus V)& \arrow{l} A^2(E)   \\
L^2(D)\arrow{r}&L^2(D\setminus W) \arrow{u}{\pi^*}  \arrow{r}{P_{D\setminus W}}  &A^2(D\setminus W) \arrow{u}{\pi^*} & \arrow{l} A^2(D)
\end{tikzcd}
\]
where the new horizontal arrows are restriction operators. The ones on the left are unitary isomorphisms because proper analytic subsets have measure zero. The ones on the right are also unitary isomorphism by a well-known removable singularity theorem (see, e.g., p. 687 of \cite{bell_1982}). It is now easy to define an operator $L^2(D)\rightarrow L^2(E)$, which we keep denoting by $\pi^*$, such that $\pi^*(A^2(D))\subseteq A^2(E)$ and the diagram 
\begin{equation}\label{bell_diagram}
\begin{tikzcd}[row sep=huge]
L^2(E)  \arrow{r}{P_{E}}  &A^2(E)   \\
L^2(D) \arrow{u}{\pi^*}  \arrow{r}{P_{D}}  &A^2(D) \arrow{u}{\pi^*} 
\end{tikzcd}
\end{equation}
commutes. This is Bell's theorem for ramified coverings. If $f\in L^2(D)$, then $\pi^*f$ is the $L^2$ function on $E$ coinciding with $\pi^*(f_{|D\setminus W})$ on $E\setminus V$ (where the last $\pi^*$ is the ``genuine'' pull-back operator \eqref{pull-back}). 

The range of $\pi^*$ has a neat description in the case of normal ramified coverings. This is the content of the following proposition, which proves the second part of Proposition \ref{prp_lifting_0}.

\begin{prp}[Cf. \cite{bender_chakrabarti_edholm_mainkar}]\label{bell_proper_cor_2} Let $\pi:E\dashrightarrow D$ be a normal ramified covering with covering group $G$. Then, 
\[
\pi^*(L^2(D))=\{u\in L^2(E)\colon \ u=g^* u \quad \forall g\in G\}.
\]
($g^*$ induces a unitary isomorphism of $L^2(E\setminus V)$, and hence of $L^2(E)$). Denote by $\Pi_G: L^2(E)\to \pi^*(L^2(D))$ the orthogonal projection from $L^2(E)$ onto $\pi^*(L^2(D))$. Then

\begin{equation}\label{G-average}
\Pi_G(u)=\frac{1}{|G|}\sum_{g\in G} g^*u.
\end{equation}
\end{prp}

\begin{proof}
If $v\in L^2(D\setminus W)$ and $u=\pi^*v\in L^2(E\setminus V)$, then the chain rule and the identity $\pi\circ g=\pi$ give
\[
g^*u=J(g) \cdot (J(\pi)\cdot v\circ \pi)\circ g = J(\pi\circ g)\cdot v\circ (\pi\circ g)=u.
\]
By the way $\pi^*:L^2(D)\rightarrow L^2(E)$ is defined, the inclusion $\subseteq$ follows. 

For the converse inclusion it is convenient to work with everywhere defined functions and not just equivalence classes modulo null sets. Hence, given an element $u\in L^2(E)$ such that $g^*u=u$ for every $g\in G$, we may find a set of full measure $A\subseteq E\setminus V$ and a square-integrable function $u_0:A\rightarrow \C$ in the equivalence class $u$ such that the identity $J(g)\cdot u_0\circ g=u_0$ holds pointwise for every $g\in G$. Notice that the existence of such a set $A$ is not problematic since the group is finite. We may also assume that $A$ is $G$-invariant, that is, a union of fibers of the covering $\pi:E\setminus V\rightarrow D\setminus W$. For this it is enough to replace $A$ with $\bigcap_{g\in G}g.A$. Notice that, by \eqref{change_of_variable}, $B:=\pi(A)\subseteq D\setminus W$ has full measure. Let now $V=\{z\in E\colon \ J(\pi)(z)=0\}$ be the branch locus of the map $\pi$. If $z\in B$, then $J(\pi)(w)\neq 0$ for every $w$ in $\pi^{-1}(z)$ (because $\pi:E\setminus V\rightarrow D\setminus W$ is a local biholomorphism), and we may define 
\[
v_0(z):=(J(\pi)(w))^{-1}u_0(w),
\] 
where $w\in E$ is any point in $\pi^{-1}(z)$. We now prove that $v_0$ is a well-defined function. Since the covering is normal, if $w_1$ and $w_2$ are two distinct points in $\pi^{-1}(z)$, then $w_2=g.w_1$ for some $g\in G$ and $J(g)(w_1) u_0(w_2)=u_0(w_1)$, because $w_1\in A$ by construction. Hence, \[
 (J(\pi)(w_1))^{-1}u_0(w_1)=J(g)(w_1)(J(\pi)(w_1))^{-1}u_0(w_2)=(J(\pi)(w_2))^{-1}u_0(w_2),
\]
as we wished to show. Let $v$ be the (almost everywhere) equivalence class of functions represented by $v_0$. Of course, $J(\pi)\cdot v\circ \pi=u$ and $v\in L^2(D)$ since $u\in L^2(E)$. 

Finally, Formula \eqref{G-average} is standard, because $u\mapsto g^*u$ is a unitary representation of the finite group $G$ on $L^2(E)$ and $\pi^*( L^2(D))$ is the subspace of $G$-invariant vectors. 
\end{proof}

We now complete the proof of Proposition \ref{prp_lifting_0} and prove Proposition \ref{prp:estimate_average}.

\begin{proof}[Proof of Proposition \ref{prp_lifting_0}]
Assume that \eqref{Bergman_E_0} holds. For $f\in L^2(D)$ we have $P_E(\pi^*f) = \pi^*P_Df$ by Bell's theorem, i.e. the diagram \eqref{bell_diagram}. Hence, if $d$ is the degree of the covering, \begin{align*}
\int_D|P_Df|^p&=\frac{1}{d}\int_{E\setminus V}|(P_Df)\circ\pi|^p |J(\pi)|^2\\
&=\frac{1}{d}\int_E|\pi^*P_Df|^p\sigma\\
&=\frac{1}{d}\int_{E}|P_E (\pi^*f)|^p\sigma\\
&\leq C \frac{1}{d}\int_E|\pi^*f|^p\sigma\\
&= C\frac{1}{d}\int_{E\setminus V}|f\circ \pi|^p |J(\pi)|^2\\
&=C\int_D|f|^p.
\end{align*}

The proof of the converse implication is similar, and hence omitted.\end{proof}

\begin{proof}[Proof of Proposition \ref{prp:estimate_average}
]
The inequality \eqref{Bergman_E_0} is clearly equivalent to \begin{equation}\label{QPi} \int_E|\Qa v|^p\sigma\leq C \int_E|\Pi_G v|^p\sigma\qquad \forall v\in L^2(E)\cap L^p(E, \sigma).
\end{equation}
Since $\Qa\circ \Pi_G=\Qa$, \eqref{Q_estimate} implies \eqref{QPi}. To see the converse implication, notice that \begin{eqnarray*}
&&\int_E|g^*v|^p\sigma=\int_E|(v\circ g)\cdot J(g)|^p|J(\pi)|^{2-p}=\int_E|v|^p|J(\pi\circ g^{-1})|^{2-p}
=\int_E|v|^p\sigma.
\end{eqnarray*}
Hence, $||\Pi_G(v)||_{L^p(E, \sigma)}\leq ||v||_{L^p(E, \sigma)}$ and we are done.

By self-adjunction of the projection $\Pi_G$ we have \[
\Qa v(\cdot)=\int_E \Pi_Gv(w)\overline{K(w,\cdot)} = \int_E  v(w)\overline{\Pi_{G,w} K(w,\cdot)}=\int_E  v(w) K_G(\cdot,w), 
\]
where $\Pi_G$ acts in the variable $w$ in the second integral. \end{proof}

By a simple averaging argument and the classical transformation rule for the Bergman kernel under the action of biholomorphism we also obtain the alternative formula 
\[
K_G(z,w)=\frac{1}{|G|}\sum_{g\in G}K(g.z,w)J(g),
\]
and the more symmetric formula
\[
K_G(z,w)=\frac{1}{|G|^2}\sum_{g,h\in G}J(g) K(g.z,h.w)\overline{J(h)}.
\]

\section{Finite Unitary Reflection Groups and Ramified Coverings}\label{section-groups}

Here we discuss in more detail the basic properties of f.u.r.g.'s needed later, some of which have already been anticipated in the introduction.  

Let $G$ be a f.u.r.g. and let $\C[z_1,\ldots, z_n]^G$ be its ring of invariants.

\begin{thm}[\cite{shephard_todd, chevalley, lehrer_taylor}]
There exist $n$ homogeneous algebraically independent polynomials $P_1,\ldots, P_n$ such that $\C[z_1,\ldots, z_n]^G$ is generated as a $\C$-algebra by $P_1,\ldots, P_n$.
\end{thm}

Recall that the ring of invariants of any finite group acting on $\C^n$ is finitely generated (\cite[Corollary 3.8]{lehrer_taylor}). By the theorem above, f.u.r.g.'s have the remarkable property that their ring of invariants is a polynomial algebra on exactly $n$ generators. For any choice of $n$ homogeneous and algebraically independent generators $P_1,\ldots, P_n$, we have the associated \emph{$G$-orbit map} $\pi(z)=\big(P_1(z),\ldots P_n(z)\big)$, which induces a bijection $\C^n/G\rightarrow \C^n$ (\cite[Proposition 9.3]{lehrer_taylor}). 

Recall from the introduction that $\mathcal R_G$ denotes the set of all the reflecting hyperplanes of $G$ and that for $Y\in \mathcal{R}_G$ we have the finite (cyclic) subgroup $G_Y$ of order $m_Y$. Finally, recall that if $Y\in \mathcal{R}_G$, then any vector $e$ of norm $1$ orthogonal to $Y$ is called a root of the reflection $r$.

\begin{prp}[{\cite[Theorem 9.8]{lehrer_taylor}}]\label{prp_Jacobian}
Choose a root $e_Y$ for every reflecting hyperplane $Y$ of $G$. Then the Jacobian of the orbit map $\pi=(P_1,\ldots,P_n):\mathbb C^n\to\mathbb C^n$ is given by 
\begin{equation*}
J(\pi)(z) = c_\pi\prod_{Y\in\mathcal R_G}\left\langle z, e_Y\right\rangle^{m_Y-1},
\end{equation*}
where $c_\pi\in\mathbb C$ is a nonzero constant depending on the choice of generators and $\left\langle\cdot, \cdot\right\rangle$ is the standard Hermitian product of $\C^n$. 
\end{prp}

We point out that, up to a multiplicative constant, the Jacobian of a $G$-orbit map coincides with the Jacobian polynomial defined in \eqref{J_G}. Moreover, the Jacobian polynomial $J_G(z)$ has the following remarkable property. 

\begin{prp}[{\cite[Lemma 9.10]{lehrer_taylor}}]\label{algebraic_division_lem}
Let $P\in \C[z_1,\ldots, z_n]$ be a skew polynomial, that is, \[
P(g.z)=\det(g)^{-1}P(z)\qquad \forall g\in G. 
\]
Then $J_G$ divides $P$. 
\end{prp}

Notice that in the skewness condition above $\det(g)^{-1}=\overline{\det(g)}$, because the group is unitary. 
\medskip

We conclude this section with the proof of Proposition \ref{rudin_criterion_0}, which is a straightforward generalization of the argument used by Rudin \cite{rudin_refl} for the case $E=B_n$, the unit ball of $\C^n$. 

\begin{proof}[Proof of Proposition \ref{rudin_criterion_0}]
The $G$-orbit map $\pi:\C^n\rightarrow \C^n$ separates $G$-orbits \cite[Theorem 3.5]{lehrer_taylor} (in fact, as we already pointed out above, $\pi$ induces a bijection of $\C^n/G$ onto $\C^n$). From this and the $G$-invariance of $E$ it follows that the sets $\pi(bE)$ and $\pi(E)$ are disjoint. Since $\pi(bE)$ is compact, the connected components of $\C^n\setminus \pi(bE)$ are open. Pick $z_0\in E$, and let $D$ be the connected component of $\C^n\setminus \pi(bE)$ containing $\pi(z_0)$. By connectedness, $\pi(E)\subseteq D$, and then $\pi(bE)\subseteq bD$ (because $\pi(bE)\subseteq \overline{D}$ by continuity). Thus, $\pi:E\rightarrow D$ is a proper holomorphic map and hence a ramified covering (see Remark \ref{rmk:proper}). The normality of the covering and the fact that the covering group is $G$ is an easy consequence of the fact that $\pi$ separates $G$-orbits.
\end{proof}

\section{Division Lemma}\label{section-division-lemma}

In this section we prove the first of the lemmas from which Theorem \ref{Lp_bdd} will follow. Let $G$ be a fixed f.u.r.g. We need the following standard terminology. If $E$ is a domain, we say that $f\in \mathcal{O}(E)$ can be analytically continued past a boundary point $p\in bE$ if there exists an open ball $B$ centered at $p$ such that $f$ extends to a holomorphic function on $E\cup B$. 

\begin{lem}[\textbf{Division Lemma}]\label{division_lemma}
Let $E$ be a good bounded $G$-invariant domain containing the origin and let $K(z,w)$ be its Bergman kernel. Then $K(z,\overline w)$ is a holomorphic function on \[\widetilde{E}=\{(z,\overline{w})\colon z,w\in E\}.\] Suppose that it can be analytically continued past any boundary point in the set  \[b\widetilde{E}\setminus\{(z,\overline z)\colon\ z\in bE\}.\] Then the averaged kernel \eqref{averaged_kernel} can be represented as
\[
K_G(z,\overline w)=J_G(z)\overline{J_G(\overline w)}M(z,w),
\]
where $M\in \mathcal{O}(\widetilde{E})$. Moreover, $M(z,w)$ can be analytically continued past any boundary point of $b\widetilde{E}\setminus\{(g.z,\overline z)\colon\ z\in bE,\, g\in G\}$. \end{lem}

The rest of this section is devoted to the proof of Lemma \ref{division_lemma}. It relies on three propositions. 

The first proposition is a well-known global divisibility property in the ring $\mathcal{O}(E)$, stated in the only special case that we need, that is, when the divisor function is a linear form. We include an elementary proof (not relying on nontrivial properties of the ring $\mathcal{O}(E)$) for the sake of completeness. 

\begin{prp}\label{prp_divisibility}
Let $E\subseteq\mathbb C^n$ be a domain, $f\in \mathcal{O}(E)$, and $\ell$ a linear form on $\C^n$. Assume that $E\cap\{\ell=0\}$ is connected and that $\ell|f$ on some open set $E'\subseteq E$ such that $E'\cap \{\ell=0\}$ is not empty. Then $\ell|f$ on $E$. 
\end{prp}

Here, given $f,h\in \mathcal{O}(E)$ and an open subset $E'\subseteq E$, we say that $h$  divides $f$ on $E'\subseteq E$ (in short, $h|f$ on $E'$) if there exists $q\in \mathcal{O}(E')$ such that $f_{|E'}=qh_{|E'}$. 

\begin{rmk}
If $E\cap \{\ell=0\}$ is not connected, local divisibility does not necessarily propagate to global divisibility. For an example of this situation, one may consider the worm-like domain \[
E=\{(z_1,z_2)\in \C^2\colon\ \pi<|z_1|<5\pi,\ |z_2-e^{i|z_1|}|<1\}
\]
and the function $f$ obtained by analytic continuation of the germ $\log(z_2)$ at the point $(2\pi, 1)$ (here $\log$ is the principal determination of the logarithm), which vanishes on $\{z_2=1\}$ near $(2\pi, 1)$, but not near $(4\pi, 1)$. 
\end{rmk}

\begin{proof}[Proof of Proposition \ref{prp_divisibility}]
Without loss of generality, we may consider the case $\ell(z)=z_1$. We first notice that if $z_1|f$ on a family of domains $E_\alpha \subseteq E$, then $z_1|f$ on $\bigcup_\alpha E_\alpha$ (if $f=z_1 q_\alpha$ on $E_\alpha$ and $f=z_1 q_\beta$ on $E_\beta$, observe that $q_\alpha=q_\beta$ on $E_\alpha\cap E_\beta\cap \{z_1\neq 0\}$ and argue by continuity). Hence, taking the union of all the domains where the division $z_1|f$ holds we obtain the maximal domain $E'\subseteq E$ with respect to this property.

We claim that $E'=E$. Clearly $E\cap\{z_1\neq 0\}\subseteq E'$, so that to verify our claim it is enough to prove that $E'\cap\{z_1=0\}= E\cap\{z_1=0\}$. We argue by contradiction: if this is not the case, then there is a boundary point $p$ of $E'\cap\{z_1=0\}$ in the interior of $E\cap\{z_1=0\}$ (with respect to the relative topology of $\{z_1=0\}$). Choose a sequence $p_n\rightarrow p$ such that $p_n\in E'\cap \{z_1=0\}$. Since $f=z_1q$ on $E'$ for some $q\in \mathcal{O}(E')$, \[
\partial_z^\alpha f(p_n) = (z_1\partial_z^\alpha q)(p_n)=0\quad \forall n, \quad \forall \alpha\in\mathbb N^n\colon\ \alpha_1=0.
\]
Thus, by continuity,  $\partial^\alpha f(p)=0$ for every $\alpha$ as above. Considering now the Taylor series centered at $p$ of $f$ we immediately see that $z_1|f$ on a small ball $B$ centered at $p$, contradicting the maximality of $E'$.
\end{proof}

Our second proposition is a generalization of Proposition \ref{algebraic_division_lem} from polynomials to functions holomorphic on appropriate domains. 

\begin{prp}\label{lem-continuatio}
Let $E$ be a good bounded $G$-invariant domain containing the origin. Let $f\in \mathcal{O}(E)$ be a skew holomorphic function, i.e., \[
f(g.z)=\det(g)^{-1}f(z) \quad \forall z\in E,\, g\in G.
\]
Then $J_G|f$ on $E$. Moreover, if $f$ can be analytically continued past a boundary point, then the quotient $f/J_G$ can also be continued past that boundary point. 
\end{prp}

\begin{proof}
Let $B$ a ball centered at $0$ where $f$ has Taylor expansion $\sum_m f_m$, with $f_m$ a homogeneous polynomial of degree $m$.
Since $f$ is skew, then every $f_m$ is skew, and hence $J_G|f_m$ by Lemma \ref{algebraic_division_lem}. Let $\ell$ be a linear factor of $J_G$, i.e., a linear form defining one of the reflecting hyperplanes. Since $\ell| f_m$ for every $m$, $\ell|f$ on $B$ (this is easily seen changing variables linearly so that $\ell(z)=z_1$). By Proposition \ref{prp_divisibility}, $\ell|f$ on the whole domain $E$. If $\ell'$ is another linear factor (possibly equal to $\ell$) of $J_G$, we may iterate this argument to conclude that $\ell'|\frac{f}{\ell}$. Continuing this way, we get that $J_G|f$ on $E$. 

To prove the second part of the statement we argue as follows. Let $p\in bE$ and let $B$ be a ball centered at $p$ where $f$ can be analytically continued. Choosing $B$ small enough, we may assume that $B$ intersects only those $Y\in \RH_G$ where $p$ lies. Consider the domain $E_1= E\cup B$ and notice that $E_1\cap Y$ is connected for every $Y\in \RH_G$, as the union of two connected sets with nonempty intersection (here one uses the fact that $E$ is good). Thus, applying the argument of the first part of the lemma to the domain $E_1$, we conclude  that $J_G|f$ on $E_1$. Notice that the fact that $E_1$ is not $G$-invariant is not a problem, since all we used
is the skewness of $f$ on a small ball centered at the origin (alternatively, one may enlarge $E_1$ saturating it w.r.t. the $G$-action).
\end{proof} 

The last ingredient we need is the next elementary proposition. 

\begin{prp}\label{prp_good}
Let $G_i\subseteq U(n_i)$ be a f.u.r.g.\ and $E_i\subseteq \C^{n_i}$ a good bounded $G_i$-invariant domain ($i=1,2$). Then, $E_1\times E_2\subseteq \C^{n_1}\times \C^{n_2}$ is a good $(G_1\times G_2)$-invariant domain. Here $G_1\times G_2=\{g_1\oplus g_2\colon\ g_i\in G_i\ (i=1,2)\}\subseteq U(n_1+n_2)$. 
\end{prp}

\begin{proof}
The family of reflecting hyperplanes $\mathcal R_{G_1\times G_2}$ consists of 
\[
Y_1\oplus \C^{n_2}, \quad \C^{n_1}\oplus Y_2,
\] 
as $Y_1$ and $Y_2$ vary in $\mathcal R_{G_1}$ and $\mathcal R_{G_2}$ respectively. The connectedness of $Y\cap (E_1\times E_2)$ for every $Y\in \mathcal{R}_{G_1\times G_2}$ is clear. 

If $(z,w)\in \C^{n_1}\times \C^{n_2}$ is a boundary point of $E_1\times E_2$, then either $z\in bE_1$ and $w\in \overline{E_2}$ or $z\in \overline{E_1}$ and $w\in bE_2$. Of course, it is enough to consider the first case. If $(z,w)\in Y\in\mathcal R_{G_1\times G_2}$, then either $z\in Y_1$ for some $Y_1\in \mathcal{R}_{G_1}$ or $w\in Y_2$ for some $Y_2\in \mathcal{R}_{G_2}$. In the first case, since $E_1$ is good and $z\in bE_1\cap Y_1$, $z$ is a boundary point of $Y_1\cap E_1$. As a consequence, $(z,w)$ is in the boundary of $(Y_1\oplus \C^n)\cap (E_1\times E_2)=(Y_1\cap E_1)\times E_2$ as we wanted to show. Assume now that $w\in Y_2$. Either $w\in E_2$, or $w\in bE_2$, and, since $E_2$ is good, in any case $w\in \overline{Y_2\cap E_2}$. Hence, $(z,w)$ is in the boundary of $(\C^n\oplus Y_2)\cap (E_1\times E_2)= E_1\times(E_2\cap Y_2)$.\end{proof}

\begin{proof}[Proof of Lemma \ref{division_lemma}]

Let $E$ be a good bounded $G$-invariant domain containing the origin with Bergman kernel $K$. Recall that \[
\widetilde{E}=\{(z,\overline{w})\colon z,w\in E\}\subseteq\C^n\times \C^n
\] and that we are assuming that the kernel $K(z,\overline{w})$ can be analytically continued past any point of $b\widetilde{E}\setminus \{(z,\overline z)\colon\ z\in bE\}$. The domain $\widetilde{E}$ is invariant under the action of the unitary group \[
G\times \overline{G}=\{g\oplus \overline{h}\colon g,h\in G\} \subseteq U(2n),
\] where $\overline{h}$ is the matrix obtained conjugating componentwise $h$. 

One may easily verify that $\overline{G}=\{\overline{h}\colon\ h\in G\}$ is a f.u.r.g. whose Jacobian polynomial is $J_{\overline{G}}(z)=\overline{J_G(\overline{z})}$. Since the Jacobian polynomial of a product of f.u.r.g.'s is the tensor product of their Jacobian polynomials, we have $J_{G\times \overline{G}}(z,w)=J_G(z)\overline{J_G(\overline{w})}$. 

By Proposition \ref{prp_good}, the domain $\widetilde{E}$ is good (because the "conjugated" domain $\{\overline{w}\colon\ w\in E\}$ is good). The function
\[
L(z,w):=K_G(z,\overline{w})=\frac{1}{|G|^2}\sum_{g,h\in G}\det(g)\overline{\det(h)}K(g.z, h.\overline{w})
\] is holomorphic and skew on $\widetilde{E}$. Indeed,
\begin{eqnarray*}
L(g'.z,\overline{h'}.w)&=&K_G(g'.z,h'.\overline{w})=\frac{1}{|G|^2}\sum_{g,h\in G}\det(g)\overline{\det(h)}K(gg'.z, hh'.\overline{w})\\
&=&\overline{\det(g')}\det(h')\frac{1}{|G|^2}\sum_{g,h\in G}\det(g)\overline{\det(h)}K(g.z, h.\overline{w})\\
&=&\det(g'\oplus \overline{h'})^{-1}L(z,w).
\end{eqnarray*}
So, by Proposition \ref{lem-continuatio}, the Jacobian polynomial of the group $G\times \overline{G}$ divides $L$, i.e., there exists $M\in \mathcal{O}(\widetilde{E})$ such that \[
L(z,w)=J_G(z)\overline{J_G(\overline{w})} M(z,w),
\]
that is \[
K_G(z,w)=J_G(z)\overline{J_G(w)} M(z,\overline{w}).
\]

By definition \eqref{averaged_kernel}, $L(z,w)$ can be analytically continued past every point of $b\widetilde{E}\setminus \{(g.z,\overline{z})\colon \ z\in bE,\, g\in G\}$, Lemma \ref{lem-continuatio} guarantees that $M(z,w)$ can be analytically continued past every such point too. The proof of Lemma \ref{division_lemma} is now complete.
\end{proof}

\section{Normal Subgroup Lemma}\label{section-normal-lemma}

In this section we discuss our second key lemma. As above, fix a f.u.r.g. $G$. 

\begin{lem}[\textbf{Normal Subgroup Lemma}]\label{normal_subgroup_lemma}

Let $\mathcal S\subseteq \mathcal R_G$ be a $G$-invariant set of reflecting hyperplanes of $G$, let $H$ be the normal reflection subgroup of $G$ generated by $\mathcal S$ (as described in Section \ref{section_main_thm})  and assume that $\mathcal R_{H}=\mathcal S$ (cf. Remark \ref{RGS=S}). For $\delta>0$, we define
\begin{equation}\label{good_region}
E(\mathcal{S},\delta)=\left\{(z,w)\in E\times E\colon \ d(z,Y),\ d(w,Y) \geq \delta\quad\forall Y\notin \mathcal{S}\right\}.
\end{equation}
Here $d(\cdot, Y)$ is the Euclidean distance from $Y$. Given $p\in (1,+\infty)$, let\[
K_{G,p}(z,w) = |J_G(z)|^{\frac{2}{p}-1}K_G(z,w)|J_G(w)|^{1-\frac{2}{p}}.
\]
Define $K_{H,p}$ similarly in terms of the averaged kernel $K_H$ and the Jacobian polynomial $J_H$ of the subgroup $H$. Then there exists a positive constant $C=C(\delta, p)$ such that for every $(z,w)\in E(\mathcal{S},\delta)$ we have \[
|\Kw(z,w)|\leq  \frac{C}{|G|}\sum_{g\in G}|K_{H,p}(g.z,w)|.
\]
\end{lem}

\begin{proof} The proof is based on two identities. A standard averaging argument gives the first one: \begin{equation}\label{averaging}
K_G(z,w)=\frac{1}{|G|}\sum_{g\in G} K_H(g.z,w)\det g\quad \forall z,w\in E.
\end{equation}
In fact, \eqref{averaging} holds for every subgroup $H$, not necessarily normal. 

Next, observe that the multiplicities $m_Y$ of $Y\in \mathcal{R}_H$ as a reflecting hyperplane of $G$ and $H$ are the same. The second identity states that for every $g\in G$ we have
\begin{equation}\label{J_invariance}
|J_H(g.z)|=|J_H(z)| \quad \forall z\in\mathbb C^n.
\end{equation}
This follows from the definition \eqref{J_G} and the easily checked fact that, since $g\in G$ normalizes $H$, the map $Y\mapsto g.Y$ is a permutation of $\mathcal{R}_H$ preserving the multiplicities $m_Y$ (recall that since $H$ is normal, $\mathcal R_H$ is $G$-invariant). The modulus is needed because of the nonuniqueness of the roots. 

From the considerations above, we see that there is $c=c(\delta)>0$ such that, if $d(z,Y)=|\left\langle z, e_Y\right\rangle|\geq \delta$ for every $Y\notin \mathcal{R}_H$, then
\begin{align*}
    |J_G(z)|&=\big|\prod_{Y\in\mathcal R_G}\left\langle z, e_Y\right\rangle^{m_Y-1}\big|\geq c \big|\prod_{Y\in\mathcal{R}_H}\left\langle z, e_Y\right\rangle^{m_Y-1}\big|    =|J_H(z)|.
\end{align*}
Since $E$ is a bounded domain we also have the converse inequality $|J_G(z)|\leq C' |J_H(z)|$ for all $z\in E$ (where $C'<+\infty$). Therefore, for all $(z,w)\in E(\mathcal{S},\delta)$, we have
\begin{eqnarray*}
|\Kw(z,w)|&=&|J_G(z)|^{\frac{2}{p}-1}|J_G(w)|^{1-\frac{2}{p}}|K_G(z,w)|\\
&\leq& C(\delta,p) |J_H(z)|^{\frac{2}{p}-1}|J_H(w)|^{1-\frac{2}{p}}|K_G(z,w)|\\
&=&C(\delta,p)|J_H(w)|^{1-\frac{2}{p}}\frac{1}{|G|}\Big|\sum_{g\in G}\det(g)\,|J_H(z)|^{\frac{2}{p}-1}K_H(g.z,w)\Big|\\
&=&C(\delta,p)|J_H(w)|^{1-\frac{2}{p}}\frac{1}{|G|}\Big|\sum_{g\in G}\det(g)\,|J_H(g.z)|^{\frac{2}{p}-1}K_H(g.z,w)\Big|\\
&\leq &\frac{C(\delta,p)}{|G|}\sum_{g\in G}\,|K_{H,p}(g.z,w)|,
\end{eqnarray*}
where we used the two identities \eqref{averaging} and \eqref{J_invariance}.\end{proof}

\section{A covering lemma in \texorpdfstring{$\mathbb C^2$}{\space} }\label{section-covering}

The Division Lemma and the Normal Subgroup Lemma of the previous sections hold for general f.u.r.g.'s. Our third and last lemma is restricted to the two-dimensional setting.

\begin{lem}[\textbf{Covering Lemma}]\label{bad_good_lemma} Let $G\subseteq U(2)$ be a f.u.r.g., and let
\[
\RH_G=\mathcal{S}_1
\cup\mathcal{S}_2\] be a partition of $\mathcal{R}_G$ into $G$-invariant subsets. Let $E\subseteq \C^n$ be a bounded $G$-invariant domain containing the origin. Let $E(\mathcal{S}_j,\delta)$ ($j=1,2$) be the regions defined in \eqref{good_region}, and let \[
E_{\mathrm{reg}}(\delta):=\{(z,w)\in E\times E\colon \ d(z,bE)+d(w,bE)+|g.z-w|\geq \delta\quad \forall g\in G\}.
\]
Then there exists $\delta>0$ such that  
\[ E\times E=E(\mathcal{S}_1,\delta)\cup E(\mathcal{S}_2,\delta)\cup E_{\mathrm{reg}}(\delta).
\]
\end{lem}

\begin{proof} We argue by “compactness and contradiction''.
Assume that such a $\delta$ does not exist. Then, for any fixed sequence $\delta_k\rightarrow 0$, there exists a sequence $\{(z_k,w_k)\}_k$ such that the following properties hold:
\begin{itemize}
\item $(z_k,w_k)\rightarrow (z,w)\in \overline{E}\times \overline{E}$ 
\item $(z_k,w_k)\notin E(\mathcal{S}_1,\delta_k)\cup E(\mathcal{S}_2,\delta_k)\cup E_{\mathrm{reg}}(\delta_k)$ for every $k$. 
\end{itemize}

Since $(z_k,w_k)\notin E_{\mathrm{reg}}(\delta_k)$, we have $(z,w)\in bE\times bE$ and, since $G$ is finite, we may assume that $g.z=w$ for some $g\in G$. 

Moreover, by the condition $(z_k,w_k)\notin E(\mathcal{S}_1,\delta_k)$ and the finiteness of $\mathcal{R}_G$, we may assume that there is $Y'\in \mathcal{S}_2$ such that either $d(z_k, Y')<\delta_k$ for all $k$ or $d(w_k, Y')<\delta_k$ for all $k$. Hence, either $z\in Y'$ or $w\in Y'$. Using $g.z=w$ and the $G$-invariance of $\mathcal{S}_2$ (since $H_2$ is normal in $G$), we find $Y\in \mathcal{S}_2$ such that $z\in Y$. Similarly, by the condition $(z_k,w_k)\notin E(\mathcal{S}_2,\delta_k)$, we may find $Z\in \mathcal{S}_1$ such that $z\in Z$. Since the two distinct lines $Y$ and $Z$ intersect only at the origin, this contradicts the fact that $z\in bE$. 
\end{proof}

\section{Proof of the Main Theorem}\label{section-main}

We finally come to the proof of Theorem \ref{Lp_bdd}. By the Covering Lemma \ref{bad_good_lemma}, there exists $\delta>0$ such that \[
E\times E = E(\mathcal{S}_1,\delta)\cup E(\mathcal{S}_2, \delta)\cup E_{\mathrm{reg}}(\delta). 
\]

Since $\mathcal{R}_{G_j}=\mathcal{S}_j$ ($j=1,2$), the Normal Subgroup Lemma \ref{normal_subgroup_lemma} allows to bound $K_{G,p}(z,w)$ on $E(\mathcal{S}_1,\delta)$ and $E(\mathcal{S}_2,\delta)$: \[
|\Kw(z,w)|\leq  \frac{C(\delta, p)}{|G|}\sum_{g\in G}|K_{G_j,p}(g.z,w)| \qquad \forall (z,w)\in E(\mathcal{S}_j,\delta),\quad (j=1,2).
\]

By the Division Lemma \ref{division_lemma}, \[
|K_{G,p}(z,w)| = |J_G(z)|^{\frac{2}{p}}|M(z,\overline{w})|\, |J_G(w)|^{2-\frac{2}{p}}\leq C(p)|M(z,\overline{w})|, 
\]
because $p\in (1,+\infty)$. Writing \[
E_{\mathrm{reg}}(\delta) = \left\{d(z,bE)\geq \frac{\delta}{3}\right\}\cup \left\{ d(w,bE)\geq \frac{\delta}{3}\right\} \cup \left\{|g.z-w|\geq \frac{\delta}{3}\quad \forall g\in G \right\}, 
\]
we see that $E_{\mathrm{reg}}(\delta)$ is at positive distance from the singular set of $M(z,\overline{w})$. Thus $K_{G,p}(z,w)$ is uniformly bounded on $E_{\mathrm{reg}}(\delta)$. All in all, \[
|K_{G,p}(z,w)|\leq C(\delta, p)\frac{1}{|G|}\sum_{g\in G}\left(|K_{G_1,p}(g.z,w)|+|K_{G_2,p}(g.z,w)|+1\right).
\]

\begin{appendix}\section{}\label{appendix}

\begin{thm}\label{thm_appendix}
Let $B_2$ be the unit ball in $\mathbb C^2$ and let $\pi: B_2\dashrightarrow \pi(B_2)=:D$ be the normal ramified covering defined by
\[
\pi(z_1,z_2)=(z_1^2,z_2)
\]
with associated covering group $G=\{r,\id\}$, where $r(z_1,z_2)=(-z_1,z_2)$. Then, the integral operator with positive kernel $|K_{G,p}|$ is $L^p(B_2)$-bounded for every $p\in(1,\infty)$.
\end{thm}

\begin{proof}
Up to an immaterial positive multiplicative constant, we have
\begin{align*}
    K_{G,p}(z,w)&= |z_1|^{\frac 2p-1}\Big(\frac{1}{(1-z_1\overline w_1-z_2\overline w_2)^3}-\frac{1}{(1+z_1\overline w_1-z_2\overline w_2
    )^3}\Big)|w_1|^{1-\frac 2p}, 
\end{align*}
where we are writing $z=(z_1,z_2)$ and $w=(w_1,w_2)$. 

Set  $\mathcal G=\{(z,w)\in B_2\times B_2: |z_1 \overline w_1|\geq\frac{1}{2}|1-z_2\overline w_2|\}$. In this region \begin{eqnarray*}
\frac{|z_1|}{|w_1|}\leq 2\frac{|z_1|^2}{|1-z_2\overline{w_2}|}\leq 2\frac{1-|z_2|^2}{(1-|z_2||w_2|)}\leq 4
\end{eqnarray*}
and an identical bound holds for $\frac{|w_1|}{|z_1|}$. Thus, on $\mathcal G$ the factor $(|z_1|/|w_1|)^{\frac 2p-1}$ is bounded and 
\begin{equation}\label{kernel_estimate_G}
|K_{G,p}(z,w)|\leq C(p) \Big(|K_{B_2}(z,w)|+|K_{B_2}(z,rw)|
\Big)
\end{equation}
where $K_{B_2}$ is the Bergman kernel of the unit ball.

On the complement $(B_2\times B_2)\backslash \mathcal G$ we take advantage of the power series expansion
\[
\frac{1}{(1-\alpha)^3}=\sum_{k=0}^{+\infty}\frac{(k+1)(k+2)}{2}\alpha^k \qquad (|\alpha|<1), 
\]  which allows to write
\begin{align*}
\frac{1}{(1-z_1\overline w_1-z_2\overline w_2)^3}=\frac{1}{(1-z_2\overline w_2)^3}\sum_{k=0}^{\infty}\frac{(k+1)(k+2)}{2} \left(\frac{z_1\overline w_1}{1-z_2\overline w_2}\right)^k
\end{align*}
and an analogous identity for the term $(1+z_1\overline w_1-z_2\overline w_2)^{-3}$. We get
\begin{align*}
K_{G,p}(z,w)&=|z_1|^{\frac 2p-1}|w_1|^{1-\frac 2p}\frac{1}{(1-z_2\overline w_2)^3}\sum_{k=0}^{\infty}\frac{(k+1)(k+2)}{2}\frac{(z_1\overline w_1)^{k}}{(1-z_2\overline w_2)^k}\big(1-(-1)^{k}\big)\\
&=|z_1|^{\frac 2p-1}|w_1|^{1-\frac 2p}\frac{z_1\overline w_1}{(1-z_2\overline w_2)^4}\sum_{k=0}^{\infty}(2k+2)(2k+3)\frac{(z_1\overline w_1)^{2k}}{(1-z_2\overline w_2)^{2k}}.
\end{align*}
Since $C:=\sum_{k=0}^{\infty}\frac{(2k+2)(2k+3)}{4^k}<+\infty$, we obtain
\begin{equation}\label{bound_G}
|K_{G,p}(z,w)|\leq C \frac{|z_1|^{\frac 2p}|w_1|^{2-\frac 2p}}{|1-z_2\overline w_2|^4}\qquad \forall (z,w)\notin \mathcal{G}.
\end{equation}
Notice that both powers in the numerator are positive. Hence, we may use $|z_1|^2\leq 1-|z_2|^2\leq 2|1-z_2\overline w_2|$
and the similar estimate $|w_1|^2\leq 2|1-z_2\overline w_2|$,
to bound the RHS of \eqref{bound_G} with a constant times $\frac{1}{|1-z_2\overline w_2|^3}$.

We finally observe that, on the complement of $\mathcal{G}$,
$|1-z_1\overline w_1-z_2\overline w_2|\leq \frac{3}{2}|1-z_2\overline w_2|$
which allows to conclude that on this set we have
\begin{equation}\label{kernel_estimate_Gc}
|K_{G,p}(z,w)|\leq C\frac{1}{|1-z_1\overline w_1-z_2\overline w_2|^3}=C'|K_{B_2}(z,w)|.   
\end{equation}
The conclusion follows from \eqref{kernel_estimate_G}
 and \eqref{kernel_estimate_Gc} and the well-known $L^p$-boundedness, for $p\in(1,\infty)$, of the integral operator with kernel $|K_{B_2}(z,w)|$ (see, e.g., Chapter 7 of \cite{rudin_ftub}).
 \end{proof}
\end{appendix}
  \bibliography{Bergman-bib}
 \bibliographystyle{amsalpha}
\end{document}